\documentclass[11pt]{article}

\title{ECH capacities and the Ruelle invariant}
\author{Michael Hutchings}
\date{}

\usepackage{amssymb}
\usepackage{latexsym}
\usepackage{amsmath}
\usepackage{amsthm}
\usepackage{amscd}

\usepackage[dvips]{graphics}
\usepackage{overpic}
\usepackage[mathscr]{euscript}

\newcommand{\mc}[1]{{\mathcal #1}}

\numberwithin{equation}{section}

\newtheorem{theorem}{Theorem}[section]
\newtheorem{proposition}[theorem]{Proposition}
\newtheorem{corollary}[theorem]{Corollary}
\newtheorem{lemma}[theorem]{Lemma}
\newtheorem{lemma-definition}[theorem]{Lemma-Definition}

\newtheorem{conjecture}[theorem]{Conjecture}

\theoremstyle{definition}
\newtheorem{definition}[theorem]{Definition}
\newtheorem{remark}[theorem]{Remark}

\newtheorem{example}[theorem]{Example}

\newcommand{\floor}[1]{\left\lfloor #1 \right\rfloor}
\newcommand{\ceil}[1]{\left\lceil #1 \right\rceil}

\newcommand{\C}{{\mathbb C}}
\newcommand{\Q}{{\mathbb Q}}
\newcommand{\R}{{\mathbb R}}

\newcommand{\Z}{{\mathbb Z}}

\newcommand{\op}{\operatorname}

\newcommand{\bpm}{\begin{pmatrix}}
\newcommand{\epm}{\end{pmatrix}}

\renewcommand{\epsilon}{\varepsilon}

\begin{document}

\maketitle

\begin{abstract}
The ECH capacities are a sequence of real numbers associated to any symplectic four-manifold, which are monotone with respect to symplectic embeddings. It is known that for a compact star-shaped domain in $\R^4$, the ECH capacities asymptotically recover the volume of the domain. We conjecture, with a heuristic argument, that generically the error term in this asymptotic formula converges to a constant determined by a ``Ruelle invariant'' which measures the average rotation of the Reeb flow on the boundary. Our main result is a proof of this conjecture for a large class of toric domains. As a corollary, we obtain a general obstruction to symplectic embeddings of open toric domains with the same volume. For more general domains in $\R^4$, we bound the error term with an improvement on the previously known exponent from $2/5$ to $1/4$.
\end{abstract}

\setcounter{tocdepth}{2}

\section{Introduction}

\subsection{Asymptotics of ECH capacities}

Given a symplectic $4$-manifold $(X,\omega)$, possibly noncompact or with boundary, there is associated a sequence of real numbers
\begin{equation}
\label{eqn:increasing}
0 = c_0(X,\omega) < c_1(X,\omega) \le c_2(X,\omega) \le \cdots \le\infty,
\end{equation}
called the {\em ECH capacities\/} of $(X,\omega)$. These were defined in \cite{qech} using embedded contact homology;
see \cite{bn} for a survey. Some basic properties of ECH capacities proved in \cite{qech} are:
\begin{itemize}
\item
(Monotonicity) If there exists a symplectic embedding of $(X,\omega)$ into $(X',\omega')$ then \begin{equation}
\label{eqn:monotone}
c_k(X,\omega) \le c_k(X',\omega')
\end{equation}
for all $k$.
\item
(Conformality) If $r>0$ then
\begin{equation}
\label{eqn:conformality}
c_k(X,r\omega) = rc_k(X,\omega).
\end{equation}
\item
(Disjoint unions)
Given a (possibly finite) sequence\footnote{In \cite{qech} it was assumed that the sequence of symplectic manifolds $\{(X_i,\omega_i)\}$ is finite, and in that case one has `max' instead of `sup' in \eqref{eqn:du}. The countable case follows directly from the finite case using the definition of ECH capacities in \cite{qech}.} of symplectic $4$-manifolds $\{(X_i,\omega_i)\}$, we have
\begin{equation}
\label{eqn:du}
c_k\left(\coprod_i(X_i,\omega_i)\right) = \sup_{\sum_ik_i=k}\sum_ic_{k_i}(X_i,\omega_i).
\end{equation}
\item
(Balls)
If $a>0$, define the ball
\[
B(a) = \left\{z\in\C^2\;\big|\; \pi|z|^2\le a\right\}.
\]
Then
\begin{equation}
\label{eqn:ckball}
c_k(B(a)) = da
\end{equation}
where $d$ is the unique nonnegative integer such that
\[
d^2+d\le 2k\le d^2+3d.
\]
\item
(Volume property)
If $X$ is a compact domain in $\R^4$ with piecewise smooth boundary, then
\begin{equation}
\label{eqn:vp}
\lim_{k\to\infty} \frac{c_k(X)^2}{k} = 4\op{vol}(X).
\end{equation}
\end{itemize}
Here for domains in $\R^4=\C^2$ we always take the restriction of the standard symplectic form
\[
\omega = \sum_{i=1}^2dx_i\,dy_i.
\]

The symplectic embedding obstructions resulting from the monotonicity property \eqref{eqn:monotone} are sharp in some cases, for example when $X$ and $X'$ are ellipsoids in $\R^4$, as shown by McDuff \cite{mcd}, or more generally when $X$ is a ``concave toric domain'' and $X'$ is a ``convex toric domain'', as shown by Cristofaro-Gardiner \cite{cg}. 

Define a ``nice star-shaped domain'' to be a compact domain in $\R^4$ whose boundary is smooth and transverse to the radial vector field.
If $X$ is a nice star-shaped domain, then the asymptotic formula \eqref{eqn:vp} is a special case of a more general result about the asymptotics of the ``ECH spectrum'' of a contact three-manifold, which was proved in \cite{vc} using Seiberg-Witten theory. The formula \eqref{eqn:vp} for nice star-shaped domains corresponds to the case when the contact three-manifold is the boundary of $X$, which of course is diffeomorphic to $S^3$, together with an induced contact form (see \eqref{eqn:Liouville} below) whose kernel is the tight contact structure.

The ECH spectrum of a contact three-manifold is defined in terms of the periods of certain Reeb orbits, and as a result the asymptotic formula for the ECH spectrum has various applications to dynamics. In particular, \cite{two} deduces the existence of at least two simple Reeb orbits; \cite{infinity} proves the existence of either two or infinitely many simple Reeb orbits under certain hypotheses; \cite{irie1,ai} obtain $C^\infty$ generic density of Reeb orbits and periodic orbits of Hamiltonian surface diffeomorphisms, see also the survey \cite{humiliere}; and \cite{calabi,weiler} obtain relations between periodic orbits of area preserving disk or annulus diffeomorphisms and the Calabi invariant.

Returning to symplectic embedding problems, the asymptotic formula \eqref{eqn:vp} implies that for $k$ large, the symplectic embedding obstruction \eqref{eqn:monotone} recovers the obvious volume constraint $\op{vol}(X)\le \op{vol}(X')$. Additional embedding obstructions arise from the deviation of $c_k(X)$ from the asymptotics in \eqref{eqn:vp}. More precisely, define the ``error term''
\begin{equation}
\label{eqn:errorterm}
e_k(X) = c_k(X) - 2\sqrt{k\op{vol}(X)}
\end{equation}
It is then interesting to try to understand the size of this error term and its geometric significance.

A result of Sun \cite{sun} implies that if $X$ is a nice star-shaped domain, then
\[
e_k(X) = O\left(k^{125/252}\right).
\]
The exponent was improved by Cristofaro-Gardiner and Savale \cite{cgs} to $2/5$. Both of these results for nice star-shaped domains are special cases of general results on the asymptotics of the ECH spectrum of a contact three-manifold, proved using Seiberg-Witten theory.

We use more elementary arguments to further improve the exponent for domains in $\R^4$:

\begin{theorem}
\label{thm:exponent}
(proved in \S\ref{sec:exponent})
If $X$ is a compact domain in $\R^4$ with smooth boundary (not necessarily star-shaped), then
\[
e_k(X) = O\left(k^{1/4}\right).
\]
\end{theorem}

In fact, $e_k(X)$ is $O(1)$ in all examples for which it has been computed.

\begin{example}
\label{ex:ball}
Let $X$ be the ball $B(a)$. We have $\op{vol}(B(a))=a^2/2$, see \eqref{eqn:volarea} below. By \eqref{eqn:ckball}, we then have
\[
e_k(B(a)) = \left(d-\sqrt{2k}\right)a,
\]
where $d$ is the unique nonnegative integer such that
\[
d^2+d\le 2k\le d^2+3d.
\]
It follows from the above two lines that
\begin{align}
\label{eqn:deviation}
\lim\inf_{k\to\infty} e_k(B(a)) &= -\frac{3}{2}a,\\
\nonumber
\lim\sup_{k\to\infty} e_k(B(a)) &= -\frac{1}{2}a.
\end{align}
\end{example}

More generally, \cite[Thm.\ 1.1]{benw} implies that for certain ``lattice convex toric domains'', $e_k$ is also $O(1)$ with a more complicated oscillating behavior.

\subsection{The Ruelle invariant}

We now formulate a general conjecture about the limiting behavior of the error term $e_k$. This requires a digression to define the ``Ruelle invariant'' of a contact form on a homology three-sphere, which can be regarded as a measure of the average rotation rate of the Reeb flow. (One can also define the Ruelle invariant more generally for volume-preserving vector fields.)

Let $\widetilde{\op{Sp}}(2)$ denote the universal cover of the group $\op{Sp}(2)$ of $2\times 2$ symplectic matrices. There is a standard ``rotation number'' function
\[
\op{rot}:\widetilde{\op{Sp}}(2)\longrightarrow \R
\]
defined as follows. Let $A\in\op{Sp}(2)$, and let $\widetilde{A}\in\widetilde{\op{Sp}}(2)$ be a lift of $A$, represented by a path $\{A_t\}_{t\in[0,1]}$ in $\op{Sp}(2)$ with $A_0=I$ and $A_1=A$. If $v$ is a nonzero vector in $\R^2$, then the path of vectors $\{A_tv\}_{t\in[0,1]}$ rotates by some angle which we denote by $2\pi\rho(v)\in\R$. We then define
\[
\op{rot}\left(\widetilde{A}\right) = \lim_{n\to\infty}\frac{1}{n}\sum_{k=1}^n\rho\left(A^{k-1}v\right).
\]
This does not depend on the choice of nonzero vector $v$. For example, if $A$ is conjugate to rotation by angle $2\pi\theta$, then $\op{rot}\left(\widetilde{A}\right)$ is a lift of $\theta$ from $\R/2\pi\Z$ to $\R$. The rotation number is a quasimorphism: if $\widetilde{B}$ is another element of $\widetilde{\op{Sp}}(2)$, then
\begin{equation}
\label{eqn:quasimorphism}
\left|\op{rot}\left(\widetilde{A}\widetilde{B}\right) - \op{rot}\left(\widetilde{A}\right) - \op{rot}\left(\widetilde{B}\right)\right| < 1.
\end{equation}

Now let $Y$ be a homology three-sphere, and let $\lambda$ be a contact form on $Y$ with associated contact structure $\xi$ and Reeb vector field $R$. For $t\in\R$, let $\phi_t:Y\to Y$ denote the diffeomorphism given by the time $t$ Reeb flow. For each $y\in Y$, the derivative of $\phi_t$ restricts to a linear map
\begin{equation}
\label{eqn:dphit}
d\phi_t:\xi_y \longrightarrow \xi_{\phi_t(y)}
\end{equation}
which is symplectic with respect to $d\lambda$. Now fix a symplectic trivialization of $\xi$, consisting of a symplectic linear map $\tau:\xi_y\to\R^2$ for each $y\in Y$. Then for $y\in Y$ and $t\in\R$, the composition
\[
\R^2\stackrel{\tau^{-1}}{\longrightarrow} \xi_y \stackrel{d\phi_t}{\longrightarrow} \xi_{\phi_t(y)} \stackrel{\tau}{\longrightarrow} \R^2
\]
is a symplectic matrix which we denote by $A^\tau_{y,t}$. In particular, if $y\in Y$ and $T\ge 0$, then the path of symplectic matrices $\{A^\tau_{y,t}\}_{t\in[0,T]}$ defines an element of $\widetilde{\op{Sp}}(2)$. We denote its rotation number by
\[
\op{rot}_\tau(y,T) = \op{rot}\left(\{A^\tau_{y,t}\}_{t\in[0,T]}\right) \in \R.
\]

As explained by Ruelle \cite{ruelle}, see also \cite[\S3.2]{gg}, one can use the quasimorphism property \eqref{eqn:quasimorphism} to show that for almost all $y\in Y$, the limit
\[
\rho(y)=
\lim_{T\to\infty}\frac{1}{T}\op{rot}_\tau(y,T)
\]
is well defined and independent of $\tau$, and the function $\rho$ is integrable.

\begin{definition}
If $Y$ is a homology three-sphere and $\lambda$ is a contact form on $Y$, define the {\em Ruelle invariant\/}
\begin{equation}
\label{eqn:Rurot}
\op{Ru}(Y,\lambda) = \int_Y\rho\,\lambda\wedge d\lambda.
\end{equation}
\end{definition}

If $X$ is a nice star-shaped domain in $\R^4$, then the standard Liouville form
\begin{equation}
\label{eqn:Liouville}
\lambda_0 = \frac{1}{2}\sum_{i=1}^2\left(x_i\,dy_i - y_i\,dx_i\right)
\end{equation}
restricts to a contact form on $\partial X$.

\begin{definition}
If $X$ is a nice star-shaped domain in $\R^4$, then we define
\[
\op{Ru}(X) = \op{Ru}\left(\partial X,{\lambda_0}|_{\partial X}\right).
\]
\end{definition}

We can now state our main conjecture:

\begin{conjecture}
\label{conj:main}
If $X$ is a generic nice star-shaped domain in $\R^4$, then
\begin{equation}
\label{eqn:conj}
\boxed{
\lim_{k\to \infty}e_k(X) = -\frac{1}{2}\op{Ru}(X).
}
\end{equation}
\end{conjecture}

\begin{example}
The ball $B(a)$ from Example~\ref{ex:ball} does not satisfy the above conjecture (hence the word ``generic'' in the conjecture), since $e_k(B(a))$ does not coverge. However we will see below that $\op{Ru}(B(a))=2a$, so it is still true that $(-1/2)\op{Ru}(B(a))$ is between the lim inf and lim sup of $e_k(B(a))$. One might conjecture that for any nice star-shaped domain, not necessarily generic, $e_k$ is $O(1)$ and the Ruelle invariant is between the lim inf and the lim sup.
\end{example}

\subsection{Results for toric domains}

Given a domain $\Omega$ in the nonnegative quadrant of $\R^2$, we define an associated {\em toric domain\/}
\[
X_\Omega = \left\{z\in\C^2\;\big|\;\pi(|z_1|^2,|z_2|^2)\in\Omega\right\}.
\]
The factor of $\pi$ ensures among other things that
\begin{equation}
\label{eqn:volarea}
\op{vol}(X_\Omega) = \op{area}(\Omega).
\end{equation}

\begin{definition}
A {\em nice toric domain\/} is a toric domain $X_\Omega$ which is also a nice star-shaped domain, meaning that $\partial X_\Omega$ is a smooth hypersurface transverse to the radial vector field. This implies that $\partial\Omega$ consists of the line segment from $(0,0)$ to $(a,0)$ for some $a>0$, the line segment from $(0,0)$ to $(0,b)$ for some $b>0$, and a smooth curve from $(0,b)$ to $(a,0)$ which is transverse to the radial vector field on $\R^2$. We denote the numbers $a$ and $b$ by $a(\Omega)$ and $b(\Omega)$, and the smooth curve from $(0,b)$ to $(a,0)$ by $\partial_+\Omega$.
\end{definition}

\begin{example}
Suppose $\Omega$ is the triangle with vertices $(0,0)$, $(a,0)$, and $(0,b)$. Then $X_\Omega$ is the ellipsoid
\[
E(a,b) = \left\{z\in\C^2\;\bigg|\; \frac{\pi|z_1|^2}{a} + \frac{\pi|z_2|^2}{b}\le 1\right\}.
\]
This is a nice toric domain.
\end{example}

\begin{definition}
\label{def:sctd}
A {\em strictly convex toric domain\/} is a nice toric domain $X_\Omega$ in which $\partial_+\Omega$ is the graph of a function $f:[0,a]\to[0,b]$ with $f(0)=b$, $f'(0) < 0$, $f''<0$ everywhere, and $f(a)=0$.

A {\em strictly concave toric domain\/} is a nice toric domain $X_\Omega$ in which $\partial_+\Omega$ is the graph of a function $f:[0,a]\to[0,b]$ with $f(0)=b$, $f''>0$ everywhere, and $f(a)=0$.
\end{definition}

We can now state one of the main results of this paper:

\begin{theorem}
\label{thm:main}
(proved in \S\ref{sec:main})
Equation~\eqref{eqn:conj} holds whenever $X$ is a strictly convex or strictly concave toric domain\footnote{It is shown in \cite{looijenga} that Theorem~\ref{thm:main} generalizes to (not necessarily strictly) convex and concave toric domains such that $\partial_+\Omega$ has no edges of rational slope.}.
\end{theorem}

To clarify what this theorem says, we have:

\begin{proposition}
\label{prop:toricRuelle}
(proved in \S\ref{sec:Rutoric})
Let $X_\Omega$ be a nice toric domain such that $\partial_+\Omega$ has negative slope\footnote{For nice toric domains in $\R^4$, the condition that $\partial_+\Omega$ has negative slope is equivalent to dynamical convexity by \cite[Prop.\ 1.8]{examples}. In fact, the negative slope hypothesis can be removed from Proposition~\ref{prop:toricRuelle} by a more careful argument \cite{gz}.} everywhere. Then
\[
\op{Ru}(X_\Omega) = a(\Omega) + b(\Omega).
\]
\end{proposition}

\begin{remark}
Equation~\eqref{eqn:conj} also holds for ellipsoids $E(a,b)$ with $a/b$ irrational, by \cite[Lem.\ 2.2]{cgls}.
\end{remark}

It is quite possible that equation \eqref{eqn:conj} is special to toric domains and that Conjecture~\ref{conj:main} is false more generally. Nonetheless, the toric case already gives an application to symplectic embedding problems:

\begin{corollary}
\label{cor:folding}
Let $X_\Omega$ and $X_{\Omega'}$ be nice toric domains satisying \eqref{eqn:conj}, e.g.\ strictly convex or strictly concave toric domains, or irrational ellipsoids. Suppose that $\op{vol}(X_\Omega)=\op{vol}(X_{\Omega'})$ and that there exists a symplectic embedding $\op{int}(X_\Omega)\to X_{\Omega'}$. Then
\[
a(\Omega) + b(\Omega) \ge a(\Omega') + b(\Omega').
\]
\end{corollary}

\begin{proof}
The interior of $X_\Omega$ has the same ECH capacities as $X_\Omega$; see \cite[\S4.2]{qech}. Thus, by the monotonicity of the ECH capacities \eqref{eqn:monotone}, the definition of the error term \eqref{eqn:errorterm}, and the hypothesis that $\op{vol}(X_\Omega)=\op{vol}(X_{\Omega'})$, we have
\[
e_k(X_\Omega) \le e_k(X_{\Omega'})
\]
for all $k$. Since $X_\Omega$ and $X_{\Omega'}$ satisfy \eqref{eqn:conj}, it follows from Proposition~\ref{prop:toricRuelle} that
\[
\frac{-(a(\Omega) + b(\Omega))}{2} \le \frac{-(a(\Omega') + b(\Omega'))}{2}.
\]
\end{proof}

\begin{remark}
Corollary~\ref{cor:folding} is not vacuous; there are examples of symplectic embeddings of an open toric domain into another (nonsymplectomorphic) toric domain of the same volume, including many cases when the domains are ellipsoids. For example, it is shown in \cite{mcds} that if $a\ge (17/6)^2$, then the interior of the ellipsoid $E(1,a)$ symplectically embeds into a ball\footnote{Although Corollary~\ref{cor:folding} is not applicable here because the ball does not satisfy \eqref{eqn:conj}, the conclusion of Corollary~\ref{cor:folding} is still true in this example since $1+a\ge 2\sqrt{a}$.} of the same volume, namely $E(\sqrt{a},\sqrt{a})$. 
\end{remark}

\begin{remark}
The examples of nice star-shaped domains $X$ discussed here seem to have $e_k(X)$ negative for all $k>0$. However there also exist examples of nice star-shaped domains $X\subset \R^4$ with $e_1(X)$ positive. The reason is that if $X$ is a nice star-shaped domain, then by the definition of ECH capacities, $c_1(X)\ge \mc{A}_{\op{min}}(X)$, where $\mc{A}_{\op{min}}(X)$ denotes the minimum symplectic action (period) of a Reeb orbit on $\partial X$. Now define the {\em systolic ratio\/}
\[
\op{sys}(X) = \frac{\mc{A}_{\op{min}}(X)^2}{2\op{vol}(X)}.
\]
It then follows from \eqref{eqn:errorterm} that
\[
e_1(X) \le 0 \Longrightarrow \op{sys}(X) \le 2.
\]
However it is shown in \cite{abhs} that there exist nice star-shaped domains with systolic ratio greater than $2$ (in fact arbitrarily large), so these must have $e_1$ positive.

On the other hand, in the {\em dynamically convex\/} case, the best known examples \cite{abhs2} have systolic ratio $2-\epsilon$. A reasonable conjecture would be that if $X$ is dynamically convex then $e_k(X) < 0$ for all $k>0$.
\end{remark}

\subsection{Outline of the rest of the paper}

In \S\ref{sec:Rutoric} we prove Proposition~\ref{prop:toricRuelle}, computing the Ruelle invariant of some toric domains, by direct calculation.

In \S\ref{sec:main} we prove the main result, Theorem~\ref{thm:main}. To do so, we use two formulas for the ECH capacities of concave toric domains proved in \cite{concave}: one in terms of the ``weight expansion'', and one in terms of lattice paths. We also use two similar formulas for the ECH capacities of convex toric domains from \cite{cg}. By carefully estimating using all four of these formulas and combining the results with Proposition~\ref{prop:toricRuelle}, we obtain the theorem.

In \S\ref{sec:exponent} we prove Theorem~\ref{thm:exponent}. The idea is to estimate the ECH capacities of a region by packing it with cubes in a naive way. The estimates we get in this case are not as good as in the case of toric domains, because concave toric domains can be packed ``more efficiently'' with balls coming from the weight expansion.

In \S\ref{sec:heuristics} we give a heuristic discussion of why we expect Conjecture~\ref{conj:main} to be true, by comparing the definition of the ECH index to Arnold's asymptotic linking number and relating this to a conjecture by Irie on equidistribution properties of ECH capacities. While this is far from a proof, we do see the volume and Ruelle invariant emerge naturally.

\paragraph{Acknowledgments.} Thanks to Alberto Abbondandolo, Julian Chaidez, and Umberto Hryniewicz for explaining the Ruelle invariant; to Dusa McDuff for explaining Lemma~\ref{lem:mcduff}; and to Dan Cristofaro-Gardiner for discussions about the asymptotics of ECH capacities. Partially supported by NSF grant DMS-1708899 and a Humboldt Research Award.

\section{The Ruelle invariant of toric domains}
\label{sec:Rutoric}

We now prove Proposition~\ref{prop:toricRuelle}, computing the Ruelle invariant of a nice toric domain $X_\Omega$ such that $\partial_+\Omega$ has everywhere negative slope.

To start, we denote the Euclidean coordinates on the plane in which $\Omega$ lives by $\mu_1$ and $\mu_2$. Define two functions
\[
\alpha,\beta:\partial_+\Omega\longrightarrow \R
\]
as follows: Given $(\mu_1,\mu_2)\in\partial_+\Omega$, the tangent line to $\partial_+\Omega$ through $(\mu_1,\mu_2)$ intersects the axes at the points $(\alpha(\mu_1,\mu_2),0)$ and $(0,\beta(\mu_1,\mu_2))$.

Proposition~\ref{prop:toricRuelle} now follows from the two lemmas below:

\begin{lemma}
\label{lem:Rutoric1}
If $X_\Omega$ is a nice toric tomain such that $\partial_+\Omega$ has everywhere negative slope, then
\begin{equation}
\label{eqn:Rutoric1}
\op{Ru}(X_\Omega) = \int_{\partial_+\Omega}\frac{\alpha+\beta}{\alpha\beta}(\mu_1\,d\mu_2 - \mu_2\,d\mu_1)
\end{equation}
where $\partial_+\Omega$ is oriented as a curve from $(a(\Omega),0)$ to $(0,b(\Omega))$.
\end{lemma}

\begin{lemma}
\label{lem:Rutoric2}
If $\gamma$ is a differentiable plane curve from $(a,0)$ to $(0,b)$ with everywhere negative slope, where $a,b>0$, and if $\alpha$ and $\beta$ are defined as above, then
\[
\int_\gamma\frac{\alpha+\beta}{\alpha\beta}(\mu_1\,d\mu_2 - \mu_2\,d\mu_1) = a+b.
\]
\end{lemma}

\begin{proof}
Write $Y=\partial X_\Omega\subset\C^2$, and let $Y_0$ denote the set of $z\in Y$ such that $z_1,z_2\neq 0$. For $z=(z_1,z_2)\in Y_0$, write $\mu_i=\pi|z_i|^2$, and let $\theta_i$ denote the argument of $z_i$. In these coordinates, the standard Liouville form \eqref{eqn:Liouville} is given by
\begin{equation}
\label{eqn:lambdamu}
\lambda_0 = \frac{1}{2\pi}\left(\mu_1\,d\theta_1 + \mu_2\,d\theta_2\right).
\end{equation}
We have
\[
T_z Y = \op{span}\left(\partial_{\theta_1}, \partial_{\theta_2}, \alpha\partial_{\mu_1} - \beta\partial_{\mu_2}\right).
\]
Thus the contact plane $\xi_z$ is spanned by the vectors
\[
\begin{split}
V &= \mu_2\partial_{\theta_1} - \mu_1\partial_{\theta_2},\\
W &= \alpha\partial_{\mu_1} - \beta\partial_{\mu_2}.
\end{split}
\]
The Reeb vector field is then given by
\begin{equation}
\label{eqn:Reeb}
R = \frac{2\pi\left(\beta\partial_{\theta_1} + \alpha\partial_{\theta_2}\right)}{\alpha\beta}.
\end{equation}
Note here that $\lambda_0(R)=1$ because
\begin{equation}
\label{eqn:defab}
\beta\mu_1 + \alpha\mu_2 = \alpha\beta
\end{equation}
by the definition of $\alpha$ and $\beta$. Equation \eqref{eqn:defab} also implies that we have a symplectic trivialization $\tau'$ of $\xi|_{Y_0}$ given by
\[
(\tau')^{-1} = \left( V, \frac{-2\pi W}{\alpha\beta}\right).
\]

Since $R$ preserves $\mu_1$ and $\mu_2$, we have $[R,V]=1$, so in the notation \eqref{eqn:dphit} we have $d\phi_tV = V$. This implies that
\[
\op{rot}_{\tau'}(y,T)=0
\]
for all $y\in Y_0$ and $T\ge 0$. However we cannot use the trivialization $\tau'$ to compute the Ruelle invariant because this trivialization does not extend over $Y\setminus Y_0$. In particular, if $\tau$ is a trivialization of $\xi$ over all of $Y$, then as one moves around a circle in $Y_0$ in which either $\theta_1$ or $\theta_2$ rotates once around $S^1$, the vector $V$ rotates once around $S^1$ with respect to $\tau$. It follows that on $Y_0$ we have
\[
\rho = \frac{1}{2\pi}R(\theta_1+\theta_2).
\]
By equation \eqref{eqn:Reeb}, we conclude that
\begin{equation}
\label{eqn:conrot}
\rho = \frac{\alpha+\beta}{\alpha\beta}.
\end{equation}

Now by equation \eqref{eqn:lambdamu}, we have
\[
\lambda_0\wedge d\lambda_0 = \frac{1}{4\pi^2}(\mu_1\,d\mu_2 - \mu_2\,d\mu_1)\,d\theta_1\,d\theta_2
\]
on $Y_0$. So by equations \eqref{eqn:Rurot} and \eqref{eqn:conrot} we have
\[
\op{Ru}(X_\Omega) = \frac{1}{4\pi^2}\int_{Y_0} \frac{\alpha+\beta}{\alpha\beta} (\mu_1\,d\mu_2 - \mu_2\,d\mu_1)\,d\theta_1\,d\theta_2.
\] 
Integrating out $\theta_1$ and $\theta_2$ then gives \eqref{eqn:Rutoric1}.
\end{proof}

\begin{proof}[Proof of Lemma~\ref{lem:Rutoric2}.]
Choose an oriented parametrization of the curve $\gamma$ as $(\mu_1(t),\mu_2(t))$ for $t\in[t_0,t_1]$. Then
\begin{equation}
\label{eqn:math53}
\int_\gamma \frac{\alpha+\beta}{\alpha\beta}(\mu_1\,d\mu_2 - \mu_2\,d\mu_1) = \int_{t_0}^{t_1}\frac{\alpha+\beta}{\alpha\beta}\Delta dt
\end{equation}
where we use the notation
\[
\Delta = \mu_1\mu_2' - \mu_1'\mu_2.
\]
By the definition of $\alpha$ and $\beta$, we have
\[
\begin{split}
\alpha &= \Delta/\mu_2',\\
\beta &= -\Delta/\mu_1'.
\end{split}
\]
The integrand in \eqref{eqn:math53} is then
\[
\frac{\alpha+\beta}{\alpha\beta}\Delta = -\mu_1' + \mu_2'.
\]
The lemma now follows from the fundamental theorem of calculus.
\end{proof}

\section{Bounds on ECH capacities of toric domains}
\label{sec:main}

\subsection{The Ruelle invariant and the weight expansion}
\label{sec:weight}

To relate the Ruelle invariant to ECH capacities, we need to recall the definition of the ``weight expansion'' of a concave toric domain following \cite{concave}.

\begin{definition}
\label{def:concavetd}
A {\em concave toric domain\/} is a toric domain $X_\Omega$ such that
\[
\Omega=\{(\mu_1,\mu_2)\mid 0\le \mu_1\le a, \; 0\le \mu_2 \le f(\mu_1)\}
\]
where $f:[0,a]\to[0,b]$ is a convex function\footnote{This is more general than a strictly concave toric domain as in Definition~\ref{def:sctd}. For a strictly concave toric domain, the function $f$ must furthermore be smooth and strictly convex, and must satisfy additional conditions near $0$ and $a$ to ensure that $\partial X_\Omega$ is smooth.} for some $a,b>0$ with $f(0)=b$ and $f(a)=0$. Write $a(\Omega)=a$ and $b(\Omega)=b$, and denote the graph of $f$ by $\partial_+\Omega$.
\end{definition}

For $c>0$, let $\Delta(c)$ denote the triangle in the plane with vertices $(0,0)$, $(c,0)$, and $(0,c)$. Also, define an {\em integral affine transformation\/} to be a map $\R^2\to\R^2$ given by the composition of an element of $\op{SL}_2\Z$ with a translation. We say that two sets in $\R^2$ are {\em integral affine equivalent\/} if one is the image of the other under an integral affine transformation.

\begin{definition}
\label{def:te}
If $X_\Omega$ is a concave toric domain, we inductively define a canonical countable set $\mc{T}(\Omega)$ of triangles in $\R^2$ such that:
\begin{description}
\item{(i)} Each triangle in $\mc{T}(\Omega)$ is affine equivalent to $\Delta(c)$ for some $c$.
\item{(ii)} Two different triangles in $\mc{T}(\Omega)$ intersect only along their boundaries.
\item{(iii)} $\overline{\bigcup_{T\in\mc{T}(\Omega)}T} = \Omega$.
\end{description}

To start defining $\mc{T}(\Omega)$, let $c$ be the largest real number such that the triangle $\Delta(c)\subset \Omega$.

Now $\partial_+\Delta(c)$ coincides with $\partial_+\Omega$ on the line segment from $(t',c-t')$ to $(t'',c-t'')$ for some $t'\le t''$. If $t'>0$, let $\Omega'$ denote the closure of the component of $\Omega\setminus\Delta(c)$ with $\mu_1\le t'$; otherwise let $\Omega'=\emptyset$. If $t''<c$, let $\Omega''$ denote the closure of the component of $\Omega\setminus \Delta(c)$ with $\mu_1\ge t''$; otherwise let $\Omega''=\emptyset$.

Let $\phi':\R^2\to\R^2$ denote the integral affine transformation defined by
\[
\phi'(\mu_1,\mu_2) = (\mu_1,\mu_1+\mu_2-c).
\]
If $\Omega'$ is nonempty, then $X_{\phi'(\Omega')}$ is a concave toric domain. Likewise, let $\phi''$ denote the integral affine transformation defined by
\[
\phi''(\mu_1,\mu_2) = (\mu_1+\mu_2-c,\mu_2).
\]
If $\Omega''$ is nonempty then $X_{\phi''(\Omega'')}$ is a concave toric domain.

We now inductively define
\[
\mc{T}(\Omega) = \{\Delta(c)\}\cup \bigsqcup_{T\in\mc{T}(\phi'(\Omega'))}(\phi')^{-1}(T) \cup \bigsqcup_{T\in\mc{T}(\phi''(\Omega''))}(\phi'')^{-1}(T).
\]
Here we interpret the terms involving $\Omega'$ or $\Omega''$ to be the empty set when $\Omega'$ or $\Omega''$ are empty.
\end{definition}

Properties (i) and (ii) above are immediate from the construction. It also follows from the construction that each triangle in $\mc{T}(\Omega)$ is a subset of $\Omega$. One can prove the rest of property (iii) by elementary arguments with a bit more work; or as overkill one can use equation \eqref{eqn:ckw} below and the volume property of ECH capacities \eqref{eqn:vp}.

\begin{definition}
If $X_\Omega$ is a concave toric domain, choose an ordering $\mc{T}(\Omega)=\{T_1,T_2,\ldots\}$ where $T_i$ is integral affine equivalent to $\Delta(a_i)$ and $a_i\ge a_{i+1}$ for each $i$. The (possibly finite) sequence $(a_1,a_2,\ldots)$ is the {\em weight expansion\/} of $X_\Omega$, which we denote by $W(\Omega)$.
\end{definition}

The significance of the weight expansion is:

\begin{theorem}
\cite[Thm.\ 1.4 and Rmk.\ 1.6]{concave}
If $X_\Omega$ is a concave toric domain with weight expansion $W(\Omega) = (a_1,\ldots)$, then its ECH capacities are given by
\begin{equation}
\label{eqn:ckw}
c_k(X_\Omega) = c_k\left(\bigsqcup_iB(a_i)\right).
\end{equation}
\end{theorem}

Note that by properties (i)--(iii) above, we have
\[
\op{vol}(X_\Omega) = \op{area}(\Omega) = \frac{1}{2}\sum_ia_i^2.
\]

It turns out that $\sum_ia_i$ is also finite, and can be described explicitly as follows.

\begin{definition}
Given a line segment $L$ in the plane, define its {\em affine length\/} $\ell_{\op{Aff}}(L)\in\R$ as follows. Let $v=(a,b)$ be the vector given by the difference between the endpoints of $L$.
\begin{itemize}
\item
If $a/b\notin\Q\cup\{\infty\}$, define $\ell_{\op{Aff}}(L)=0$.
\item
If $a/b\in\Q\cup\{\infty\}$, let $d$ be the largest real number such that $(a/d,b/d)\in\Z^2$, and define $\ell_{\op{Aff}}(L)=d$.
\end{itemize}
If $\gamma$ is an injective continuous path in the plane including line segments $L_1,\ldots$, define its affine length
\[
\ell_{\op{Aff}}(\gamma) = \sum_i\ell_{\op{Aff}}(L_i).
\]
\end{definition}

\begin{lemma}
\label{lem:mcduff}
\cite{mcp}
If $X_\Omega$ is a concave toric domain with weight expansion $W(\Omega)=(a_1,\ldots)$, then
\begin{equation}
\label{eqn:mcd}
\sum_ia_i = a(\Omega) + b(\Omega) - \ell_{\op{Aff}}(\partial_+\Omega).
\end{equation}
\end{lemma}

\begin{proof}
Following the construction in Definition~\ref{def:te}, we inductively define a sequence of domains $\Omega_k$ for $k\ge 1$ such that $X_{\Omega_k}$ is a concave toric domain, $\Omega_k\subset \Omega_{k+1}$, and $\overline{\bigcup_{k}{\Omega_k}}=\Omega$, as follows. Using the notation of Definition~\ref{def:te}:
\begin{itemize}
\item
$\Omega_1 = \Delta(c)$.
\item
If $k>1$, then
\[
\Omega_k = \Delta(c) \cup (\phi')^{-1}(\phi'(\Omega')_{k-1}) \cup (\phi'')^{-1}(\phi''(\Omega'')_{k-1}).
\]
Here we omit the terms corresponding to $\Omega'$ or $\Omega''$ when those domains are empty.
\end{itemize}
Observe that $X_{\Omega_k}$ has a finite weight expansion with at most $2^k - 1$ terms. Moreover these are all terms in the weight expansion of $X_\Omega$; and if $S(\Omega)$ denotes the sum of the terms in the weight expansion $W(\Omega)$, then $\lim_{k\to\infty}S({\Omega_k})=S(\Omega)$.

We will prove by induction on $k$ that for every concave toric domain $X_\Omega$, we have
\begin{equation}
\label{eqn:mcdk}
S(\Omega_k) = a({\Omega_k}) + b({\Omega_k}) - \ell_{\op{Aff}}(\partial_+\Omega_k).
\end{equation}
The lemma then follows by fixing $\Omega$ and taking the limit of \eqref{eqn:mcdk} as $k\to\infty$.

If $k=1$, then both sides of equation \eqref{eqn:mcdk} are equal to $c$ above.

Now suppose that $k>1$. For simplicity we assume that both $\Omega'$ and $\Omega''$ are nonempty; the other cases work similarly. By induction we can assume that
\[
\begin{split}
S(\Omega'_{k-1}) &= a({\Omega'_{k-1}}) + b({\Omega'_{k-1}}) - \ell_{\op{Aff}}(\partial_+\Omega'_{k-1}),\\
S(\Omega''_{k-1}) &= a({\Omega''_{k-1}}) + b({\Omega''_{k-1}}) - \ell_{\op{Aff}}(\partial_+\Omega''_{k-1}).
\end{split}
\]
By construction we have
\[
\begin{split}
S(\Omega_k) &= c + S(\Omega'_{k-1}) + S(\Omega''_{k-1}),\\
a({\Omega_k}) &= c + a({\Omega''_{k-1}}),\\
b({\Omega_k}) &= c + b({\Omega'_{k-1}}).
\end{split}
\]
Combining the above equations, we obtain
\begin{equation}
\label{eqn:ctae}
S(\Omega_k) - a({\Omega_k}) - b({\Omega_k}) = -c + a({\Omega'_{k-1}}) + b({\Omega''_{k-1}}) - \ell_{\op{Aff}}(\partial_+\Omega'_{k-1}) - \ell_{\op{Aff}}(\partial_+\Omega''_{k-1}).
\end{equation}

Now observe that $\partial_+\Omega_k$ consists of the following:
\begin{itemize}
\item
The curve $(\phi')^{-1}(\partial_+\Omega'_{k-1})$ from $(0,c+b({\Omega'_{k-1}}))$ to $(a({\Omega'_{k-1}}),c-a({\Omega'_{k-1}}))$.
\item
The line segment from the latter point to $(c-b({\Omega''_{k-1}}),b({\Omega''_{k-1}}))$.
\item
The curve $(\phi'')^{-1}(\partial_+\Omega''_{k-1})$ from the latter point to $(c+a({\Omega''_{k-1}}),0)$. 
\end{itemize}
Since affine length is invariant under integral affine transformations, it follows that
\[
\ell_{\op{Aff}}(\partial_+\Omega_k) = \ell_{\op{Aff}}\left(\partial_+\Omega'_{k-1}\right) +  \left(c-a\left({\Omega'_{k-1}}\right) - b\left({\Omega''_{k-1}}\right)\right) + \ell_{\op{Aff}}\left(\partial_+\Omega''_{k-1}\right).
\]
Combining this last equation with \eqref{eqn:ctae} proves \eqref{eqn:mcdk}.
\end{proof}

As a corollary, we obtain a relation between the weight expansion and the Ruelle invariant in the strictly concave case:

\begin{corollary}
If $X_\Omega$ is a strictly concave toric domain (or more generally any concave toric domain such that $\partial_+\Omega$ does not contain any line segments of rational slope) with weight expansion $W(\Omega)=(a_1,\ldots)$, then
\begin{equation}
\label{eqn:cwe}
\sum_ia_i = a(\Omega) + b(\Omega).
\end{equation}
\end{corollary}

\begin{proof}
This follows from Lemma~\ref{lem:mcduff} because $\partial_+\Omega$ contains no line segments of rational slope, so its affine length is zero.
\end{proof}

\subsection{An estimate from the weight expansion}

\begin{lemma}
\label{lem:cs}
Let $(a_i)_{i=1,\ldots}$ be a (possibly finite) sequence of positive real nubers with $\sum_ia_i<\infty$. Write  $X=\coprod_iB(a_i)$ and $V=\op{vol}(X) = \frac{1}{2}\sum_ia_i^2$. Then
\[
\lim\sup_{k\to\infty}\left(c_k\left(X\right) - 2\sqrt{kV}\right) \le -\frac{1}{2}\sum_ia_i.
\]
\end{lemma}

\begin{corollary}
\label{cor:concaveub}
If $X_\Omega$ is a concave toric domain such that $\partial_+\Omega$ does not contain any line segments of rational slope, then
\[
\lim\sup_{k\to\infty} e_k(X_\Omega) \le -\frac{a(\Omega)+b(\Omega)}{2}.
\]
\end{corollary}

\begin{proof}
This follows from Lemma~\ref{lem:cs} by plugging in equations \eqref{eqn:errorterm}, \eqref{eqn:ckw}, \eqref{eqn:volarea}, and \eqref{eqn:cwe}.
\end{proof}

\begin{proof}[Proof of Lemma~\ref{lem:cs}.]
By equations \eqref{eqn:du} and \eqref{eqn:ckball}, we have
\begin{equation}
\label{eqn:ckmax}
c_k(X) = \sup\left\{\sum_ia_id_i \;\bigg|\;\sum_i(d_i^2+d_i)\le 2k\right\}
\end{equation}
where the $d_i$ are nonnegative integers. Now if we put the sequence $(a_i)$ in nonincreasing order, then in the above supremum, we can restrict to the case where $d_i=0$ for $i>k$. There are then only finitely many possibilities, so we can write `max' instead of `sup' in \eqref{eqn:ckmax}.

For each $k$, choose a sequence $d(k) = \{d(k)_i\}_{i=1,\ldots}$ realizing the maximum in \eqref{eqn:ckmax}. In particular, we have
\begin{gather}
\label{eqn:csck}
\sum_ia_id(k)_i = c_k(X),\\
\label{eqn:cs2k}
\sum_i(d(k)_i^2+d(k)_i) \le 2k.
\end{gather}

By \eqref{eqn:cs2k} and the Cauchy-Schwarz inequality, for each $k$ we have
\[
\sum_ia_i\sqrt{d(k)_i^2+d(k)_i}\le \sqrt{2V}\sqrt{2k}.
\]
Combining this with \eqref{eqn:csck}, we have
\begin{equation}
\label{eqn:ucs}
c_k(X) - 2\sqrt{kV} \le - \sum_ia_i\left(\sqrt{d(k)_i^2+d(k)_i}-d(k)_i\right).
\end{equation}
To complete the proof, it is enough to show that for fixed $i$ we have
\begin{equation}
\label{eqn:tctp}
\lim_{k\to\infty}d(k)_i = \infty,
\end{equation}
so that
\[
\lim_{k\to\infty}\left(\sqrt{d(k)_i^2+d(k)_i}-d(k)_i\right) = \frac{1}{2}.
\]

To prove \eqref{eqn:tctp}, suppose to the contrary that $\lim\inf_{k\to\infty}d(k)_i<\infty$. Then it follows similary to \eqref{eqn:ucs} that
\[
\lim\inf_{k\to\infty}\left(c_k(X) - 2\sqrt{k\left(V-\frac{1}{2}a_i^2\right)}\right) \le 0.
\]
Thus
\[
\lim\inf_{k\to\infty}\frac{c_k(X)^2}{k} \le 4\op{vol}\left(X\setminus B(a_i)\right).
\]
However the argument in \cite[Prop.\ 8.4]{qech} shows that $X$ satisfies the volume property \eqref{eqn:vp}, which is a contradiction.
\end{proof}

\subsection{Lattice point estimates} 

If $\Omega$ is a domain in the nonnegative quadrant of $\R^2$, define
\[
\widehat{\Omega} = \{(\mu_1,\mu_2)\in\R^2 \mid (|\mu_1|,|\mu_2|)\in\Omega\}.
\]

\begin{definition}
\label{def:ctd}
A {\em convex toric domain\/} is a toric domain $X_\Omega$ such that $\widehat{\Omega}$ is compact and convex with nonempty interior. Let $a(\Omega)$ and $b(\Omega)$ denote the intersections of $\partial\widehat{\Omega}$ with the positive $\mu_1$-axis and positive $\mu_2$-axis, and let $\partial_+\Omega$ denote the closure of the part of $\partial_\Omega$ not on the axes; this is a path from $(0,b(\Omega))$ to $(a(\Omega),0)$.
\end{definition}

We now prove the following estimate, which is similar to Corollary~\ref{cor:concaveub} but proved by different methods:

\begin{lemma}
\label{lem:convexub}
Let $X_\Omega$ be a convex toric domain such that $\partial_+\Omega$ is the graph of a strictly concave $C^2$ function\footnote{This is slighty more general than a ``strictly convex toric domain'', because $\partial X_\Omega$ might not be smooth.}. Then
\[
\lim\sup_{k\to\infty} e_k(X_\Omega) \le -\frac{a(\Omega)+b(\Omega)}{2}.
\]
\end{lemma}

To prove this lemma, we need to recall some material from \cite{beyond}. Let $\Omega$ be a domain as in Definition~\ref{def:ctd}. If $v$ is a vector in $\R^2$, define
\[
\|v\|_{\Omega}^* = \max\left\{\langle v,w\rangle \mid w\in\widehat{\Omega}\right\}.
\]
Note that $\|\cdot\|_\Omega^*$ is a norm; it is the dual of the norm with unit ball $\widehat{\Omega}$. If $\gamma:[\alpha,\beta]\to\R^2$ is a continuous, piecewise differentiable parametrized curve, define its {\em $\Omega$-length\/} by
\begin{equation}
\label{eqn:omegalength}
\ell_\Omega(\gamma) = \int_\alpha^\beta\|J\gamma'(t)\|_\Omega^*dt
\end{equation}
where $J=\begin{pmatrix}0&-1\\1&0\end{pmatrix}$.
The Wulff isoperimetric inequality \cite{bm,wulff} implies that if $\gamma$ is the boundary of a compact region $R$, then
\begin{equation}
\label{eqn:wulff}
\ell_\Omega(\gamma)^2\ge 4\op{Area}(\widehat{\Omega})\op{Area}(R),
\end{equation}
with equality if and only if $R$ is a scaling and translation of $\widehat{\Omega}$. Below we just need to know that equality holds in \eqref{eqn:wulff} when $R$ is a scaling of $\widehat{\Omega}$, which follows by direct calculation.

\begin{definition}
A {\em convex integral path\/} is a polygonal path $\Lambda$ in the nonnegative quadrant from the point $(0,b)$ to the point $(a,0)$, for some nonnegative integers $a$ and $b$, with vertices at lattice points, such that if $R$ denotes the region bounded by $\Lambda$ and the line segments from $(0,0)$ to $(a,0)$ and from $(0,0)$ to $(0,b)$, then $\widehat{R}$ is convex. Define $\mc{L}(\Lambda)$ to be the number of lattice points in $R$, including lattice points on the boundary.
\end{definition}

We now have the following theorem\footnote{The statement in \cite{beyond} looks slightly different, writing $\mc{L}(\Lambda)=k+1$ instead of $\mc{L}(\Lambda)\ge k+1$ in \eqref{eqn:ckconvex}. However this makes no difference, as any convex integral path $\Lambda$ with $\mc{L}(\Lambda)>k+1$ can be ``shrunk'' to a convex integral path with $\mc{L}(\Lambda)=k+1$ without increasing $\Omega$-length; see the proof of Lemma~\ref{lem:convexub} below.}, proved in \cite[Prop. 5.6]{beyond}, as a special case of \cite[Cor.\ A.5]{cg}:

\begin{theorem}
\label{thm:ckconvex}
Let $X_\Omega$ be a convex toric domain. Then
\begin{equation}
\label{eqn:ckconvex}
c_k(X) = \min\{\ell_\Omega(\Lambda)\mid \mc{L}(\Lambda)\ge k+1\}.
\end{equation}
Here the minimum is over convex integral paths $\Lambda$.
\end{theorem}

\begin{proof}[Proof of Lemma~\ref{lem:convexub}.]
Given a positive integer $k$, let $r$ be the smallest real number such that the scaling $r\Omega$ contains at least $k+1$ lattice points. The boundary of the convex hull of $r\Omega\cap\Z^2$ consists of a segment on the $\mu_1$-axis, a segment on the $\mu_2$-axis, and a convex integral path $\Lambda$ with $\mc{L}(\Lambda)\ge k+1$. Thus by Theorem~\ref{eqn:ckconvex}, we have
\begin{equation}
\label{eqn:iss1}
c_k(X_\Omega) \le \ell_\Omega(\Lambda).
\end{equation}

Next observe that
\begin{equation}
\label{eqn:iss2}
\ell_\Omega(\Lambda) \le \ell_\Omega(\partial_+(r\Omega)).
\end{equation}
The reason is that $\Lambda$ can be obtained from $\partial_+(r\Omega)$ by a finite sequence of operations, each of which replaces a portion of a curve by a line segment with the same endpoints. These operations do not increase $\Omega$-length since $\|\cdot\|_{\Omega}^*$ is a norm.

By the equality case of Wulff's isoperimetric inequality \eqref{eqn:wulff}, we have
\[
\ell_\Omega(\partial_+(r\Omega)) = 2\sqrt{\op{Area}(\Omega)\op{Area}(r\Omega)}.
\]
By \eqref{eqn:volarea}, we can rewrite the above as
\begin{equation}
\label{eqn:iss3}
\ell_\Omega(\partial_+(r\Omega)) = 2\sqrt{\op{vol}(X_\Omega)\op{Area}(r\Omega)}.
\end{equation}

Next, a classical result of van der Korput, see the refinement by Chaix \cite{chaix}, asserts that if $R$ is a region in the plane with $C^2$ strictly convex boundary, then
\[
\left||R\cap\Z^2| - \op{Area}(R)\right| \le 10000(1+M)^{2/3},
\]
where $M$ denotes the maximum radius of curvature of $\partial R$. Taking $\epsilon>0$ small and applying this result to $R=(r-\epsilon)\widehat{\Omega}$, with the intersections with the axes appropriately smoothed, we find that there is a constant $C$, depending only on $\Omega$ and not on the positive integer $k$, such that
\[
\op{Area}(r\Omega) \le k - \frac{r}{2}(a(\Omega)+b(\Omega)) + Cr^{2/3}.
\]
In particular, since $\op{Area}(r\Omega)=r^2\op{vol}(X_\Omega)$, we get
\[
r = \sqrt{\frac{k}{\op{vol}(X_\Omega)}} + o(\sqrt{k}).
\]
Putting this into the previous inequality, we get
\begin{equation}
\label{eqn:iss4}
\op{Area}(r\Omega) \le k - \left(\frac{a(\Omega)+b(\Omega)}{2\sqrt{\op{vol}(X_\Omega)}}\right)\sqrt{k} + o(\sqrt{k}).
\end{equation}
Combining \eqref{eqn:iss1}, \eqref{eqn:iss2}, \eqref{eqn:iss3}, and \eqref{eqn:iss4}, we obtain
\[
\begin{split}
c_k(X_\Omega) &\le 2\sqrt{\op{vol}(X_\Omega)\left(k - \left(\frac{a(\Omega)+b(\Omega)}{2\sqrt{\op{vol}(X_\Omega)}}\right)\sqrt{k} + o(\sqrt{k})\right)}\\
&= 2\sqrt{\op{vol}(X_\Omega)k} - \frac{a(\Omega)+b(\Omega)}{2} + o(1).
\end{split}
\]
By equation \eqref{eqn:errorterm}, the lemma follows.
\end{proof}

We also have a ``dual'' version of Lemma~\ref{lem:convexub} for concave toric domains.

\begin{lemma}
\label{lem:concavelb}
Let $X_\Omega$ be a concave toric domain (see Definition~\ref{def:concavetd}) such that $\partial_+\Omega$ is the graph of a strictly convex $C^2$ function\footnote{Again, this is a bit more general than a ``strictly concave toric domain''.}. Then
\[
\lim\inf_{k\to\infty}e_k(X_\Omega) \ge -\frac{a(\Omega)+b(\Omega)}{2}.
\]
\end{lemma}

\begin{proof}
This is proved similarly to Lemma~\ref{lem:convexub}, but with inequalities going in the reverse direction.

To start, there is a counterpart of Theorem~\ref{thm:ckconvex}, proved in \cite[Thm.\ 1.21]{concave}, which reads
\[
c_k(X_\Omega) = \max\{\ell_\Omega(\Lambda) \mid \mc{L}(\Lambda) \le k\}.
\]
Here $\Lambda$ is a {\em concave integral path\/}, which is a polygonal path with vertices at lattice points from $(0,b)$ to $(a,0)$ with $a,b\ge 0$ which is the graph of a convex function. In this context the $\Omega$-length $\ell_\Omega(\Lambda)$ is defined as in \eqref{eqn:omegalength}, but with the norm $\|\cdot\|_{\Omega}^*$ replaced by the ``anti-norm'' given by
\[
[v] = \min\{\langle v,w\rangle | w\in\partial_+\Omega\}.
\]
Finally, $\mc{L}(\Lambda)$ now denotes the number of lattice points in the region enclosed by $\Lambda$ and the axes, this time not including lattice points on $\Lambda$.

Given a positive integer $k$, let $r$ be the supremum of the set of real numbers such that the scaling $r\Omega$ contains at most $k$ lattice points. The boundary of the convex hull of the set of lattice points in the nonnegative quadrant but not in $(r-\epsilon)\Omega$ then consists of rays along the axes, together with a concave integral path $\Lambda$ satisfying $\mc{L}(\Lambda)\le k$. Thus
\[
c_k(X_\Omega) \ge \ell_\Omega(\Lambda).
\]
The rest of the proof now parallels the proof of Lemma~\ref{lem:convexub}.
\end{proof}

\subsection{Completing the proof of the main theorem}

\begin{proof}[Proof of Theorem~\ref{thm:main}.]
Let $X_\Omega$ be a strictly convex or strictly concave toric domain. By Proposition~\ref{prop:toricRuelle}, what we need to show is that
\begin{equation}
\label{eqn:mtnts}
\lim_{k\to\infty}e_k(X_\Omega) = -\frac{a(\Omega)+b(\Omega)}{2}.
\end{equation}
In the strictly concave case, this follows from Corollary~\ref{cor:concaveub} and Lemma~\ref{lem:concavelb}.

In the strictly convex case, by Lemma~\ref{lem:convexub}, we just need to show that
\begin{equation}
\label{eqn:jnts}
\lim\inf_{k\to\infty}e_k(X_\Omega) \ge -\frac{a(\Omega)+b(\Omega)}{2}.
\end{equation}
To do so, recall the notation $\Delta(c)$ from \S\ref{sec:weight}, and let $c$ be the smallest positive real number such that $\Omega\subset \Delta(c)$. Then $\partial_+\Omega$ intersects $\partial_+\Delta(c)$ in a unique point $(t,c-t)$. Suppose that $0<t<c$. (The cases where $t=0$ or $t=c$ are simpler and will be omitted.)

Let $\Omega'$ denote the closure of the component of $\Delta(c)\setminus\Omega$ with $\mu_1 < t$, and let $\Omega''$ denote the closure of the component of $\Delta(c)\setminus\Omega$ with $\mu_1 > t$. Define integral affine transformations $\phi',\phi'':\R^2\to\R^2$ by
\[
\begin{split}
\phi'(\mu_1,\mu_2) &= (c-\mu_1-\mu_2,\mu_1),\\
\phi''(\mu_1,\mu_2) &= (\mu_2,c-\mu_1-\mu_2).
\end{split}
\]
Then $X'=X_{\phi'(\Omega')}$ and $X''=X_{\phi''(\Omega'')}$ are concave toric domains satisfying the hypotheses of Corollary~\ref{cor:concaveub} and Lemma~\ref{lem:concavelb}, so that they satisfy \eqref{eqn:mtnts}. Observe also that
\[
\begin{split}
a(\phi'(\Omega')) &= c-b(\Omega),\\
b(\phi'(\Omega')) &= t,\\
a(\phi''(\Omega'')) &= c-t,\\
b(\phi''(\Omega'')) &= c - a(\Omega).
\end{split}
\]

By \cite[Thm.\ A.1]{cg}, we have
\begin{equation}
\label{eqn:fo1}
c_k(X_\Omega) = \inf_{k',k''\ge 0}\left(c_{k+k'+k''}(B(c)) - c_{k'}(X') - c_{k''}(X'')\right).
\end{equation}
By \eqref{eqn:mtnts} for $X'$ and $X''$ we get, as functions of $k'$ and $k''$,
\begin{equation}
\label{eqn:fo2}
\begin{split}
c_{k'}(X') &= 2\sqrt{k'\cdot\op{vol}(X')} + \frac{b(\Omega)-c-t}{2} + o(1),\\
c_{k''}(X'') &= 2\sqrt{k''\cdot\op{vol}(X'')} + \frac{a(\Omega)-2c+t}{2} + o(1).
\end{split}
\end{equation}
By \eqref{eqn:deviation}, we have
\begin{equation}
\label{eqn:fo3}
c_{k+k'+k''}(B(c)) \ge 2\sqrt{(k+k'+k'')\op{vol}(B(c))} - \frac{3c}{2} + o(1).
\end{equation}

Now since $\op{vol}(B(c)) = \op{vol}(X_\Omega) + \op{vol}(X') + \op{vol}(X'')$, by the Cauchy-Schwarz inequality (for three-component vectors) we have
\begin{equation}
\label{eqn:fo4}
\sqrt{(k+k'+k'')\op{vol}(B(c))} \ge \sqrt{k\op{vol}(X_\Omega)} + \sqrt{k'\op{vol}(X')} + \sqrt{k''\op{vol}(X'')}.
\end{equation}
Combining \eqref{eqn:fo1}, \eqref{eqn:fo2}, \eqref{eqn:fo3}, and \eqref{eqn:fo4}, we obtain
\[
\begin{split}
e_k(X_\Omega) &\ge \frac{-3c}{2} + \frac{-b(\Omega)+c+t}{2} + \frac{-a(\Omega)+2c-t}{2} + o(1)\\
&= -\frac{a(\Omega)+b(\Omega)}{2} + o(1).
\end{split}
\]
(Note that while the $o(1)$ terms in \eqref{eqn:fo2} are as functions of $k'$ and $k''$, we do get $o(1)$ terms as functions of $k$ above, since when $k$ is large, we must also have $k'$ and $k''$ large when close to the infimum in \eqref{eqn:fo1}, as in the proof of Lemma~\ref{lem:cs}.) This proves \eqref{eqn:jnts} for our strictly convex toric domain $X_\Omega$ and thus completes the proof of the theorem.
\end{proof}

\section{Improving the exponent in the general case}
\label{sec:exponent}

In this section we prove Theorem~\ref{thm:exponent}, estimating $e_k(X)$ for a general compact domain $X\subset\R^4$ with smooth boundary.

To prepare for this, if $a,b>0$, define the {\em polydisk\/}
\[
P(a,b) = \left\{z\in\C^2 \; \big| \; \pi|z_1|^2\le a^2, \; \pi|z_2|^2\le b^2\right\}.
\]
It was shown in \cite{qech} (and also follows directly from the more general Theorem~\ref{thm:ckconvex}) that the ECH capacities of a polydisk are given by
\begin{equation}
\label{eqn:ckp}
c_k(P(a,b)) = \min\left\{am+bm \;\big|\; (m+1)(n+1) \ge k+1\right\}
\end{equation}
where $m,n$ are nonnegative integers. We now need two simple estimates.

\begin{lemma}
\label{lem:ekp}
$e_k(P(a,a)) \ge -2a$ for all $k$.
\end{lemma}

\begin{proof}
For each nonnegative integer $k$, there is a unique nonnegative integer $d$ such that
\[
d^2 \le k \le d^2 + 2d.
\]
It follows from \eqref{eqn:ckp} that
\begin{equation}
\label{eqn:array}
c_k(P(a,a)) = \left\{\begin{array}{cl} (2d-1)a, & d^2 \le k \le d^2+d,\\
2da, & d^2+d < k \le d^2+2d.
\end{array}\right.
\end{equation}
On the other hand, $\op{vol}(P(a,a))=a^2$, so
\begin{equation}
\label{eqn:ekp}
e_k(P(a,a)) = c_k(P(a,a)) - 2a\sqrt{k}.
\end{equation}
In the first line of \eqref{eqn:array} we have $\sqrt{k} < d+1/2$, and in the second line of \eqref{eqn:array} we have $\sqrt{k}<d+1$. The lemma then follows from \eqref{eqn:array} and \eqref{eqn:ekp}.
\end{proof}

\begin{lemma}
\label{lem:basest}
Let $X$ be a bounded domain in $\R^4$, and suppose there are disjoint open subsets $P_1,P_2,\ldots\subset X$ such that $P_i$ is symplectomorphic to $\op{int}(P(a_i,a_i))$. Let $k$ be a positive integer. Let
\[
I_k=\left\{i\mid a_i^2 \ge \op{vol}(X)/k\right\}
\]
and write
\[
V_k = \sum_{i\in I_k}a_i^2 = \op{vol}\left(\bigcup_{i\in I_k}P_i\right).
\]
Then
\begin{equation}
\label{eqn:basest}
e_k(X) \ge -2\sqrt{2}\sum_{i\in I_k}a_i + 2\frac{(V_k-\op{vol}(X))}{\sqrt{\op{vol}(X)}}\sqrt{k}.
\end{equation}
\end{lemma}

\begin{proof}
For each $i$ define a positive real number
\[
\widehat{k_i} = \frac{a_i^2}{\op{vol}(X)}k,
\]
and define a nonnegative integer
\[
k_i = \floor{\widehat{k_i}}.
\]
Note that $k_i>0$ if and only if $i\in I_k$.

By the disjoint union property of ECH capacities \eqref{eqn:du} and the definition of the error term \eqref{eqn:errorterm}, we have
\[
\begin{split}
c_k(X) & \ge \sum_ic_{k_i}(P(a_i,a_i))\\
& = \sum_{i\in I_k}\left(2a_i\sqrt{k_i} + e_k(P(a_i,a_i))\right)\\
&= 2\sum_{i\in I_k}a_i\sqrt{\widehat{k_i}} + \sum_{i\in I_k}\left(2a_i\left(\sqrt{k_i} - \sqrt{\widehat{k_i}}\right) + e_k(P(a_i,a_i))\right).
\end{split}
\]
By the definition of $\widehat{k_i}$, we have
\[
\sum_{i\in I_k}a_i\sqrt{\widehat{k_i}} = \frac{V_k}{\sqrt{\op{vol}(X)}}\sqrt{k}.
\]
And for each $i\in I_k$, by Lemma~\ref{lem:ekp} and the fact that $k_i\ge 1$, we have
\[
2a_i\left(\sqrt{k_i} - \sqrt{\widehat{k_i}}\right) + e_k(P(a_i,a_i) \ge -2\sqrt{2}a_i.
\]
Combining the above three lines gives
\[
c_k(X) \ge -2\sqrt{2}\sum_{i\in I_k}a_i + \frac{2V_k}{\sqrt{\op{vol}(X)}}\sqrt{k}.
\]
The lemma now follows from the definition of the error term \eqref{eqn:errorterm}.
\end{proof}

\begin{proof}[Proof of Theorem~\ref{thm:exponent}.]
We first prove the inequality
\begin{equation}
\label{eqn:desest}
e_k(X) \ge -Ck^{1/4}.
\end{equation}
Here and below, $C$ denotes a positive constant which depends only on $X$, but which may change from one line to the next.

To do so, we inductively define a sequence $P_1,P_2,\ldots$ as in \eqref{eqn:basest} as follows. Step 1 is to add all open cubes whose vertices are consecutive points on the half-integer lattice $\frac{1}{2}\Z^4$ that are contained in $X$. For $n>1$, Step $n$ is to add all open cubes whose vertices are consecutive points in the scaled lattice $2^{-n}\Z^4$ that are contained in $X$ but not contained in any of the cubes added in the first $n-1$ steps. Each cube added in Step $n$ is symplectomorphic to the open polydisk $\op{int}(P(4^{-n},4^{-n}))$.

Let $X_n$ denote the closure of the union of all cubes added in Steps $1$ to $n$. Then we have
\begin{equation}
\label{eqn:vxxn}
\op{vol}(X\setminus X_n) \le C \cdot 2^{-n}.
\end{equation}
 The reason is that by construction, any point in $X\setminus X_n$ is within distance $2^{1-n}$ of $\partial X$. And since $\partial X$ is assumed smooth, it follows that the volume of the set of points within distance $d$ of $\partial X$ is at most $C\cdot d$ when $d$ is small.

Let $m_n$ denote the number of cubes obtained in Step $n$. Since these cubes are disjoint and each have volume $16^{-n}$, it follows from \eqref{eqn:vxxn} that
\begin{equation}
\label{eqn:mnbound}
m_n \le C \cdot 8^n.
\end{equation}

Now suppose that
\begin{equation}
\label{eqn:kbound}
16^n \le \frac{k}{\op{vol}(X)} < 16^{n+1}.
\end{equation}
Then in the notation of Lemma~\ref{lem:basest}, the set $I_k$ consists of the indices of the cubes added in the first $n$ steps. By \eqref{eqn:mnbound}, we have
\[
\sum_{i\in I_k}a_i \le C\cdot 2^n.
\]
And by \eqref{eqn:vxxn}, we have
\[
\frac{V_k-\op{vol}(X)}{\sqrt{\op{vol}(X)}} \ge -C\cdot 2^{-n}.
\]
Putting the above three lines into \eqref{eqn:basest} gives
\[
e_k(X) \ge -C \cdot 2^n.
\]
By \eqref{eqn:kbound}, we obtain \eqref{eqn:desest}.

To complete the proof of the theorem, we need to prove the reverse inequality
\[
e_k(X) \le C \cdot k^{1/4}.
\]
To do so, we choose a large cube $W$ containing $X$, divide the complement $W\setminus X$ into cubes as above, and use a similar agument. (Compare \cite[Prop.\ 8.6]{qech}.)
\end{proof}

\section{Heuristics for the conjecture}
\label{sec:heuristics}

We now review some facts from embedded contact homology, and then use these to give a heuristic discussion of why we expect Conjecture~\ref{conj:main} to be true.

\subsection{Facts}

We first briefly review some notions from embedded contact homology. Let $Y$ be a homology $3$-sphere, and let $\lambda$ be a nondegenerate contact form on $Y$.

\begin{definition}
An {\em ECH generator\/} is a finite set of pairs $\alpha=\{(\alpha_i,m_i)\}$ where:
\begin{itemize}
\item
The $\alpha_i$ are distinct simple Reeb orbits.
\item
The $m_i$ are positive integers.
\item
If $\alpha_i$ is hyperbolic (meaning that the linearized return map of the Reeb flow along $\alpha_i$ has real eigenvalues) then $m_i=1$.
\end{itemize}
\end{definition}
Define the {\em symplectic action\/} of $\alpha$ to be the real number
\[
\mc{A}(\alpha) = \sum_im_i\mc{A}(\alpha_i).
\]
Here $\mc{A}(\alpha_i)$ denotes the symplectic action, or period, of the Reeb orbit $\alpha_i$.

Let $\tau$ be a trivialization of the contact structure $\xi$; this trivialization exists and is unique up to homotopy by our assumption that $Y$ is a homology sphere. If $\gamma$ is a Reeb orbit, define its rotation number
\[
\theta(\gamma) = \op{rot}_\tau(y,\mc{A}(\gamma)) = \mc{A}(\gamma)\rho(y).
\]
where $y$ is a point on the image of $\gamma$.

\begin{definition}
If $\alpha=\{(\alpha_i,m_i)\}$ is an ECH generator, define\footnote{This is a special case of the general definition of the ECH index in \cite[Def.\ 3.5]{bn}. The relative first Chern class term there is not present here because we are using a global trivialization $\tau$.}  its {\em ECH index\/} to be the integer
\begin{equation}
\label{eqn:I}
I(\alpha) = \sum_im_i^2\op{sl}(\alpha_i) + \sum_{i\neq j}m_im_j\ell(\alpha_i,\alpha_j) + \sum_i\sum_{k=1}^{m_i}\left(\floor{k\theta(\alpha_i)} + \ceil{k\theta(\alpha_i)}\right).
\end{equation}
Here $\ell(\alpha_i,\alpha_j)$ denotes the linking number of $\alpha_i$ and $\alpha_j$; and $\op{sl}(\alpha_i)$ denotes the self-linking number of the transverse knot $\alpha_i$, which is the linking number of $\alpha_i$ with a pushoff in the direction $\tau$, see \cite[\S3.5.2]{geiges}.
\end{definition}

If $(Y,\xi)$ is diffeomorphic to $S^3$ with the tight contact structure, then one can define the {\em ECH spectrum\/} of $(Y,\lambda)$, which is a sequence of real numbers $c_k(Y,\lambda)$ indexed by nonnegative integers $k$. The relevance for our discussion is that if $X$ is a nice star-shaped domain in $\R^4$, then its ECH capacities are defined by
\[
c_k(X) = c_k(\partial X,{\lambda_0}|_{\partial X}).
\]
And the key fact we need to know is that
\begin{equation}
\label{eqn:realize}
c_k(Y,\lambda) = \mc{A}(\alpha),
\end{equation}
where $\alpha$ is a certain ECH generator with ECH index
\[
I(\alpha)=2k,
\]
selected by a ``min-max'' procedure using the ECH chain complex.

We now want to look at the index formula \eqref{eqn:I} more closely. To prepare for this we need a bit more background. Choose an auxiliary metric on $Y$. If $y\in Y$ and $T>0$, we can form a loop $\eta_{y,T}$ by starting with the path given by the time $t$ Reeb flow from $y$ to $\phi_T(y)$, and then appending a length-minizing geodesic from $\phi_T(y)$ back to $y$. (If this geodesic is not unique, pick one arbitrarily.) If $y_1,y_2$ are distinct, define the {\em asymptotic linking number\/} by
\[
f(y_1,y_2) = \lim_{T_1,T_2\to\infty} \frac{1}{T_1T_2} \ell(\eta_{y_1,T_1},\eta_{y_2,T_2}),
\]
when this limit exists. Here of course $\ell(\eta_{y_1,T_1},\eta_{y_2,T_2})$ is defined only when the loops $\eta_{y_1,T_1}$ and $\eta_{y_2,T_2}$ are disjoint. By a result of Arnold \cite{arnold} and Vogel \cite{vogel} (which applies to more general volume-preserving vector fields), the function $f$ is defined almost everywhere on $Y\times Y$ and integrable, and
\begin{equation}
\label{eqn:Arnold}
\int_{Y\times Y}f = \op{vol}(Y,\lambda).
\end{equation}
Here we are integrating with respect to the measure on $Y\times Y$ given by the product of the contact volume forms $\lambda\wedge d\lambda$, and we define $\op{vol}(Y,\lambda) = \int_Y \lambda\wedge d\lambda$.

For example, if $y_1$ and $y_2$ are on distinct simple Reeb orbits $\gamma_1$ and $\gamma_2$, then it follows from the definition that
\[
f(y_1, y_2) = \frac{1}{\mc{A}(\gamma_1)\mc{A}(\gamma_2)}\ell(\gamma_1,\gamma_2).
\]
If $y_1$ and $y_2$ are on the same simple Reeb orbit $\gamma$, then $f(y_1,y_2)$ is not defined; however it is natural to extend the definition in this case to set
\[
f(y_1,y_2) = \frac{1}{\mc{A}(\gamma)^2}\left(\op{sl}(\gamma) + \theta(\gamma)\right).
\]

Using the above formulas, we can rewrite the index formula \eqref{eqn:I} as
\begin{equation}
\label{eqn:Irewrite}
I(\alpha) = \sum_{i,j}m_im_j\mc{A}_i\mc{A}_jf_{i,j} - \sum_im_i^2\mc{A}_i\rho_i + \sum_i\sum_{k=1}^{m_i}\left(\floor{k\mc{A}_i\rho_i} + \ceil{k\mc{A}_i\rho_i}\right).
\end{equation}
Here we write $\mc{A}_i=\mc{A}(\alpha_i)$; we let $f_{i,j}$ denote $f(y_i,y_j)$ for $y_i$ in the image of $\alpha_i$ and $y_j$ in the image of $\alpha_j$; and $\rho_i$ denotes $\rho(y)$ for $y$ in the image of $\alpha_i$. 

\subsection{A new definition}

\begin{definition}
If $\alpha=\{(\alpha_i,m_i)\}$ is an ECH generator, then using the notation of \eqref{eqn:Irewrite}, define its {\em approximate ECH index\/} to be the real number
\begin{equation}
\label{eqn:Iapprox}
I_{\op{approx}}(\alpha) = \sum_{i,j}m_im_j\mc{A}_i\mc{A}_jf_{i,j} + \sum_im_i\mc{A}_i\rho_i.
\end{equation}
\end{definition}

We can bound the error in this approximation as follows:

\begin{lemma}
\label{lem:Iapprox}
$|I_{\op{aprox}}(\alpha) - I(\alpha)| \le \sum_im_i$.
\end{lemma}

\begin{proof}
It follows from \eqref{eqn:Irewrite} and \eqref{eqn:Iapprox} that
\[
I_{\op{approx}}(\alpha) - I(\alpha) = \sum_i\sum_{k=1}^{m_i}\left(2k\mc{A}_i\rho_i - \floor{k\mc{A}_i\rho_i} - \ceil{k\mc{A}_i\rho_i}\right).
\]
The lemma then follows since
\[
\big|2x - \floor{x} - \ceil{x}\big| < 1
\]
for every real number $x$.
\end{proof}

We can now suggestively rewrite \eqref{eqn:Iapprox} as
\begin{equation}
\label{eqn:suggestive}
I_{\op{approx}}(\alpha) = \int_{\alpha\times\alpha}f + \int_\alpha\rho
\end{equation}
where the integral is with respect to the measure given by the Reeb vector field, multiplied by $m_i$ on each orbit $\alpha_i$.

\subsection{Heuristics}

A conjecture of Irie \cite{irie2}, of which a version has been verified for convex and concave toric domains, asserts that if $\lambda$ is generic, then ECH generators $\alpha$ realizing $c_k(Y,\lambda)$ as in \eqref{eqn:realize} are {\em equidistributed\/} in $Y$ as $k\to\infty$. This means that if $U\subset Y$ is an open set, then the symplectic action of $\alpha\cap U$ divided by the symplectic action of $\alpha$ converges to $\op{vol}(U)/\op{vol}(Y)$. If we assume a very favorable version of this equidistribution, then by Lemma~\ref{lem:Iapprox} and equation \eqref{eqn:suggestive} we can approximate
\[
2k = I(\alpha) \approx I_{\op{approx}}(\alpha) \approx \frac{\mc{A}(\alpha)^2}{\op{vol}(Y,\lambda)^2}\int_{Y\times Y}f + \frac{\mc{A}(\alpha)}{\op{vol}(Y,\lambda)}\int_Y\rho.
\]
Here we are not discussing the size of the error in the approximation since this is just a heuristic. Comparing with \eqref{eqn:Rurot} and \eqref{eqn:Arnold}, we obtain
\[
2k\cdot\op{vol}(Y,\lambda) \approx \mc{A}(\alpha)^2 + \mc{A}(\alpha)\op{Ru}(Y,\lambda).
\]
Since $\mc{A}(\alpha) = c_k(Y,\lambda)$, we then get
\[
c_k(Y,\lambda) \approx \sqrt{2k\cdot\op{vol}(Y,\lambda)} - \frac{1}{2}\op{Ru}(Y,\lambda).
\]
When $X$ is a nice star-shaped domain, we have $\op{vol}(\partial X,{\lambda_0}|_{\partial X}) = 2\op{vol}(X)$ by Stokes's theorem, so we obtain
\[
c_k(X) \approx 2\sqrt{k\cdot\op{vol}(X)} - \frac{1}{2}\op{Ru}(X).
\]

\end{document}